\documentclass{siamltex}
\usepackage{amssymb,amsmath,amsfonts}
\usepackage{algorithm,algorithmic,verbatim,graphics,url}

\def\fl{\mathop {\rm fl}}
\def\op{\mathop {\rm op}}
\def\trace{\mathop {\rm trace}}

\begin{document}
\title{La Budde's Method for Computing Characteristic Polynomials}

\author{Rizwana Rehman\thanks{%
Department of Medicine (111D), VA Medical Center, 508 Fulton Street,
Durham, NC 27705, USA
({\tt Rizwana.Rehman@va.gov})}
\and
Ilse C.F. Ipsen\thanks{%
    Department of
    Mathematics, North Carolina State University, P.O. Box 8205,
    Raleigh, NC 27695-8205, USA ({\tt ipsen@ncsu.edu}, {\tt
      http://www4.ncsu.edu/{\char'176}ipsen/})}
}

\overfullrule = 0pt
\maketitle
\begin{abstract}
  La Budde's method computes the characteristic polynomial of a real
  matrix $A$ in two stages: first it applies orthogonal similarity
  transformations to reduce $A$ to upper Hessenberg form $H$, and
  second it computes the characteristic polynomial of $H$ from
  characteristic polynomials of leading principal submatrices of $H$.  
If $A$
  is symmetric, then $H$ is symmetric tridiagonal, and La Budde's
  method simplifies to the Sturm sequence method.  If $A$ is diagonal
  then La Budde's method reduces to the Summation Algorithm, a
  Horner-like scheme used by the MATLAB function \texttt{poly} to
  compute characteristic polynomials from eigenvalues.

  We present recursions to compute the individual coefficients of the
  characteristic polynomial in the second stage of La Budde's method,
  and derive running error bounds for symmetric and nonsymmetric
  matrices. We also show that La Budde's method can be more accurate
  than \texttt{poly}, especially for indefinite and nonsymmetric
  matrices $A$. Unlike \texttt{poly}, La Budde's method is not
  affected by illconditioning of eigenvalues, requires only real
  arithmetic, and allows the computation of individual coefficients.
\end{abstract}


\section{Introduction}
We present a little known numerical method for computing characteristic
polynomials of real matrices. The characteristic
polynomial of a $n\times n$ real matrix $A$ is defined as
$$p(\lambda)\equiv \det(\lambda I-A)=
\lambda^n+c_1\lambda^{n-1}+\cdots+c_{n-1}\lambda+c_n,$$
where $I$ is the identity matrix, $c_1=-\trace(A)$ and $c_n=(-1)^n\det(A)$.

The method was first introduced in
1956 by Wallace Givens at the third High Speed Computer Conference at
Louisiana State University \cite{Giv57}.  According to Givens, the
method was brought to his attention by his coder Donald La
Budde \cite[p 302]{Giv57}.
Finding no earlier reference to this method, we credit its
development to La Budde and thus name it ``La Budde's method''. 

La Budde's method consists of two stages:
In the first stage it reduces $A$ to upper Hessenberg form $H$ with
orthogonal similarity transformations, and 
in the second stage it computes the
characteristic polynomial of $H$. The latter is done by computing characteristic
polynomials of leading principal submatrices of successively larger
order.  Because $H$ and $A$ are similar, they have the same
characteristic polynomials. If $A$ is symmetric, then $H$ is a
symmetric tridiagonal matrix, and La Budde's method simplifies to
the Sturm sequence method \cite{Giv53}.
If $A$ is diagonal then La Budde's method reduces to the 
Summation Algorithm, a Horner-like scheme that 
is used to compute characteristic
polynomials from eigenvalues \cite{ReI10a}. The Summation Algorithm is the
basis for MATLAB's \texttt{poly} command, which 
computes the characteristic polynomial by applying the
Summation Algorithm to eigenvalues computed with \texttt{eig}.

We present recursions to compute the individual coefficients of the
characteristic polynomial in the second stage of La Budde's method.
La Budde's method has a number of advantages over \texttt{poly}.
First, a Householder reduction of $A$ to Hessenberg form $H$ in the first stage
is numerically stable, and it does not change 
the condition numbers \cite{IpsR07} of the coefficients $c_k$ 
with respect to changes in the matrix.
In contrast to \texttt{poly}, La Budde's
method is not affected by potential illconditioning of eigenvalues.
Second, La Budde's method allows the computation of individual 
coefficients $c_k$ (in the process, $c_1,\ldots, c_{k-1}$ are
computed as well) and is substantially faster if $k\ll n$.
This is important in the context of our quantum
physics application, where only a 
small number of coefficients are required \cite[\S 1]{IpsR07},
\cite{DeanLee,Lee:2004qd}.

Third, La Budde's method is efficient, requiring only about $5n^3$
floating point operations and real arithmetic. This is in 
contrast to \texttt{poly} which requires complex arithmetic when a real
matrix has complex eigenvalues.
Most importantly, La Budde's method can often be
more accurate than \texttt{poly}, and can even compute coefficients of
symmetric matrices to high relative accuracy.
Unfortunately we have not been able to derive
error bounds that are tight enough to predict this accuracy.

In this paper we assume that the matrices are real.  Error bounds for complex
matrices are derived in \cite[\S 6]{Rizth}.

\subsubsection*{Overview} 
After reviewing existing numerical methods for computing characteristic
polynomials in \S \ref{C4}, we introduce La Budde's method in \S \ref{C_6}.
Then we present recursions  for the second stage of La Budde's method 
and running error bounds, for symmetric matrices in \S \ref{s_sym}
and for nonsymmetric matrices in \S \ref{s_nonsym}.
In \S \ref{s_combined} we present running error bounds for both stages of
La Budde's method. We end with numerical experiments in \S \ref{s_exp} 
that compare La Budde's method to MATLAB's \texttt{poly} function and
demonstrate the accuracy of La Budde's method.

\section{Existing Numerical Methods}\label{C4}
In the nineteenth century and in the first half of the twentieth century
characteristic polynomials were often computed as a precursor to an eigenvalue 
computation.
In the second half of the twentieth century, however, Wilkinson and others
demonstrated that computing 
eigenvalues as roots of characteristic polynomials is numerically 
unstable \cite{Wil65,Wil84C}. As a consequence, characteristic polynomials
and methods for computing them
fell out of favor with the numerical linear algebra community.
We give a brief overview of these methods. They can be found in the 
books by Faddeeva \cite{Fad59}, Gantmacher \cite{Gant98}, and
Householder \cite{Hou64}.

\subsection{Leverrier's Method}
The first practical method for computing characteristic polynomials
was developed by Leverrier in $1840$. It is based on Newton's 
identities \cite[(7.19.2)]{Wil65}
$$c_1=-\trace(A),\qquad c_k=-\frac{1}{k}\trace\left(A^k+c_1A^{k-1}+\cdots+
c_{k-1}A\right),\quad 2\leq k\leq n.$$
The Newton identities can be expressed recursively as
$$c_k=-\frac{1}{k}\trace(AB_{k-1}),\qquad \mathrm{where}\quad
B_1 \equiv A+c_1I,\quad B_k \equiv AB_{k-1}+c_kI.$$
Leverrier's method and modifications of 
it have been rediscovered  by Faddeev and Sominski{\u \i},
Frame, Souriau, and Wegner, see \cite{HouB59}, and also 
Horst \cite{Hor35}.
Although Leverrier's method is expensive, with an operation count 
proportional to $n^4$, it continues to attract
attention. It has been proposed for computing $A^{-1}$, sequentially
\cite{Bing41} and in parallel \cite{Csa76}.
Recent papers have focused on different 
derivations of the method \cite{Barnett89,Hou98,Papa74}, combinatorial 
aspects \cite{Lewin94}, properties of the adjoint \cite{HelmWV93},
and expressions of $p(\lambda)$ in specific polynomial
bases \cite{Barnett96}.

In addition to its large operation count, Leverrier's method 
is also numerically unstable. Wilkinson remarks \cite[\S 7.19]{Wil65}:
\begin{quote}
  ``We find that it is common for severe cancellation to take place
  when the $c_i$ are computed, as can be verified by estimating the
  orders of magnitudes of the various contributions to $c_i.$''
\end{quote}
Wilkinson identified two factors that are responsible for the numerical 
instability of computing $c_k$: errors in the computation of the trace,
and errors in the previously computed coefficients $c_1,\ldots,c_{k-1}$.

Our numerical experiments on many test matrices  corroborate Wilkinson's
observations. We found that Leverrier's method gives inaccurate results
even for coefficients that are well conditioned.
For instance, consider the $n\times n$ matrix $A$ of all ones.
Its characteristic polynomial is $p(\lambda)=\lambda^n-n\lambda^{n-1}$,
so that  $c_2=\cdots=c_n=0$.
Since $A$ has only a single nonzero singular value $\sigma_1=n$,
the coefficients $c_k$ are well conditioned (because $n-1$
singular values are zero, the first order
condition numbers with respect to absolute changes in the
matrix are zero \cite[Corollary 3.9]{IpsR07}).
However for $n=40$, Leverrier's method computes values for
$c_{22}$ through $c_{40}$ in the range of $10^{18}$ to $10^{47}$.

The remaining methods described below have operation counts
proportional to $n^3$. 

\subsection{Krylov's Method and Variants}\label{S:SS1}
In 1931 Krylov presented a method that implicitly  tries to 
reduce $A$ to a companion matrix, whose
last column contains the coefficients of
$p(\lambda)$. Explicitly, the method constructs a matrix $K$ from
what we now call Krylov vectors: $v, Av, A^2v, \ldots$ where
$v\neq 0$ is an arbitrary vector. Let $m\geq 1$ 
be the \textit{grade} of the vector, that is the smallest index for which the 
vectors $v, Av,\ldots A^{m-1}v$ are linearly independent, but the 
inclusion of one more vector $A^mv$ makes the vectors linearly dependent.
Then the linear system
$$K x+A^mv=0,\qquad \mathrm{where}\quad
K\equiv \begin{pmatrix} v & Av &\ldots & A^{m-1}v\end{pmatrix}$$
has the unique solution $x$. 
Krylov's method solves this linear system $K x=-A^mv$ for $x$.
In the fortunate case when  $m=n$ the solution
$x$ contains the coefficients of $p(\lambda)$, and $x_i=-c_{n-i+1}$.
If $m<n$ then $x$ contains only coefficients of a divisor
of $p(\lambda)$.

The methods by Danilevski{\u \i}, Weber-Voetter, and Bryan
can be viewed as
particular implementations of Krylov's method \cite[\S 6]{Hou64}, as can
the method by Samuelson \cite{Sam42}. 

Although Krylov's method is quite general, it has a number of 
shortcomings. First Krylov vectors tend to become linearly dependent, so that
the linear system $Kx=-A^mv$ tends to be highly illconditioned.
Second, we do not know in advance the grade $m$ of the initial
vector $v$; therefore, we may end up only with a divisor of 
$p(\lambda)$. If $A$ is derogatory, i.e. some
eigenvalues of $A$ have geometric multiplicity 2 or larger, then every
starting vector $v$ has grade $m<n$, and Krylov's method
does not produce the characteristic polynomial of $A$.
If $A$ is non derogatory, then it is similar to its
companion matrix, and almost every starting vector should give
the characteristic polynomial. Still it is possible to start with a
vector $v$ of grade $m<n$, where Krylov's method fails to produce
$p(\lambda)$ for a non derogatory matrix $A$ \cite[Example 4.2]{Ham70}.

The problem with Krylov's method, as well as the related
methods by Danilevski{\u \i}, Weber-Voetter, Samuelson, Ryan and Horst 
is that they try to compute, 
either implicitly or explicitly, a similarity transformation to 
a companion matrix.
However, such a transformation only exists if $A$ is nonderogatory, and it
can be numerically stable only if $A$ is far from derogatory.
It is therefore not clear that remedies like those proposed 
for Danilevski{\u\i}'s method in
\cite{Han63}, \cite[p 36]{HouB59}, \cite{WaCh82},
\cite[\S 7.55]{Wil65} would be fruitful.

The analogue of the companion form for derogatory matrices is 
the Frobenius normal form. This is a similarity transformation 
to block triangular form, where the diagonal blocks are companion matrices.
Computing Frobenius normal forms is common in computer algebra
and symbolic computations e.g. \cite{Gies95}, but is numerically not 
viable because it requires
information about Jordan structure and is thus an illposed problem.
This is true also of Wiedemann's algorithm \cite{KalS91,Wie86},
 which works with $u^TA^ib$, where
$u$ is a vector, and can be considered a ``scalar version'' of Krylov's method.

\subsection{Hyman's method} 
Hyman's method computes the characteristic polynomial for  Hessenberg 
matrices \cite[\S 7.11]{Wil65}. The basic idea can be described
as follows. Let $B$ be a $n\times n$ matrix, and partition
$$B=\bordermatrix{&n-1&1 \cr 1& b_1^T & b_{12}\cr
n-1 &B_2 &b_2}.$$
If $B_2$ is nonsingular then
$\det(B)=(-1)^{n-1}\det(B_2)(b_{12}-b_1^TB_2^{-1}b_2)$.
Specifically, if $B=\lambda I-H$ where $H$ is an unreduced upper Hessenberg 
matrix then $B_2$ is nonsingular upper triangular, so that
$\det(B_2)=(-1)^{n-1}h_{21}\cdots h_{n,n-1}$ is just the product
of the subdiagonal elements. Thus
$$p(\lambda)=h_{21}\cdots h_{n,n-1}(b_{12}-b_1^TB_2^{-1}b_2).$$
The quantity $B_2^{-1}b_2$ can be computed as the solution of a
triangular system. However $b_1$, $B_2$, and $b_2$ are functions of $\lambda$.
To recover the coefficients of $\lambda^i$ requires the solution of $n$ 
upper triangular systems \cite{MQVP95}.

A structured backward error bound 
under certain conditions has been derived in \cite{MQVP95}, and
iterative refinement is suggested for improving backward accuracy.
However, it is not clear that this will help in general. The numerical
stability of Hyman's method depends on the condition number with
respect to inversion of the triangular matrix $B_2$. Since the
diagonal elements of $B_2$ are $h_{21},\ldots,h_{n,n-1}$, $B_2$ can be
ill conditioned with respect to inversion if $H$ has small subdiagonal
elements.

\subsection{Computing Characteristic Polynomials from 
Eigenvalues}\label{s_evalues}
An obvious way to compute the coefficients of the characteristic
polynomial is to compute the eigenvalues $\lambda_i$ and then
multiply out $\prod_{i=1}^n{(\lambda-\lambda_i)}$.
The MATLAB function \texttt{poly} does this. It first computes the eigenvalues
with \texttt{eig} and then uses a Horner-like scheme, the so-called
Summation Algorithm, to determine the $c_k$ from the eigenvalues $\lambda_j$
as follows:

\begin{quote}
\texttt{c = [1 zeros(1,n)]}\\
\texttt{for j = 1:n}\\
$\qquad$\texttt{c(2:(j+1)) = c(2:(j+1)) - $\lambda_j$.*c(1:j)}\\
\texttt{end}
\end{quote}

The accuracy of \texttt{poly} is highly dependent on the accuracy
with which the eigenvalues $\lambda_j$ are computed.
In \cite[\S 2.3]{ReI10a} we present perturbation bounds for characteristic
polynomials with regard to changes in the eigenvalues, and show that the
errors in the eigenvalues are amplified by elementary symmetric functions
in the absolute values of the eigenvalues. Since eigenvalues of
non-normal (or nonsymmetric) matrices are much more sensitive
than eigenvalues of normal matrices and are computed to much lower accuracy,
\texttt{poly} in turn tends to compute characteristic
polynomials of non-normal matrices to much lower accuracy.
As a consequence, \texttt{poly} gives useful results only for the limited
class of matrices with wellconditioned eigenvalues.

\section{La Budde's Method}\label{C_6}
La Budde's method works in two stages. In the first stage it reduces
a real matrix $A$ to upper Hessenberg form $H$ by 
orthogonal similarity transformations. In the second stage 
it determines the characteristic polynomial of $H$ by
successively computing characteristic polynomials of leading principal
submatrices of $H$.  Because $H$ and $A$ are similar, they have the
same characteristic polynomials. 
If $A$ is symmetric, then $H$ is a
symmetric tridiagonal matrix, and La Budde's method simplifies to 
the Sturm sequence method. The Sturm sequence
method was used by Givens \cite{Giv53} to compute
eigenvalues of a symmetric tridiagonal matrix $T$, and is the
basis for the bisection method \cite[\S\S 8.5.1, 8.5.2]{GovL96}.

Givens said about La Budde's method \cite[p 302]{Giv57}:
\begin{quote}
Since no division occurs in this second stage of the computation
  and the detailed examination of the first stage for the symmetric
  case [...] was successful in guaranteeing its accuracy there, one may hope
  that the proposed method of getting the characteristic equation will
  often yield accurate results. It is, however, probable that
  cancellations of large numbers will sometimes occur in the floating
  point additions and will thus lead to excessive
  errors.
\end{quote}

Wilkinson also preferred La Budde's method to computing
the Frobenius form. He states \cite[\S 6.57]{Wil65}:
\begin{quote}
We have described the determination of the Frobenius form in
  terms of similarity transformations for the sake of consistency and
  in order to demonstrate its relation to Danilewski's  method. 
However, since we will usually use higher precision arithmetic in the 
reduction to Frobenius form than in the reduction to Hessenberg form,
the reduced matrices arising in the derivation of the former cannot be
overwritten in the registers occupied by the Hessenberg matrix.

It is more straightforward to think in terms of a
  direct derivation of the characteristic polynomial of $H$. This
  polynomial may be obtained by recurrence relations in which we
  determine successively the characteristic polynomials of each of the
  leading principal submatrices $H_r$ ($r=1,\ldots,n$) of $H$.
[...]

 No special difficulties arise if some
  of the [subdiagonal entries of $H$] are small or even
 zero. 
\end{quote}

La Budde's method has several attractive features. 
First, a Householder reduction of $A$ to Hessenberg form $H$ in the first stage
is numerically stable \cite[\S 7.4.3]{GovL96}, \cite[\S 6.6]{Wil65}.  
Since orthogonal transformations do not change the singular values,
and the condition numbers of the coefficients $c_k$ to changes in the matrix
are functions of singular values \cite{IpsR07},
the sensitivity of the $c_k$ does not change in the reduction from $A$ to $H$.
In contrast to the eigenvalue based method in \S \ref{s_evalues}, La Budde's
method is not affected by the conditioning of the eigenvalues.

Second, La Budde's method allows the computation of individual 
coefficients $c_k$ (in the process, $c_1,\ldots, c_{k-1}$ are
computed as well) and is substantially faster if $k\ll n$.
This is important in the context of our quantum
physics application, where only a 
small number of coefficients are required \cite[\S 1]{IpsR07},
\cite{DeanLee,Lee:2004qd}.

Third, La Budde's method is efficient. The Householder reduction to 
Hessenberg form requires $10n^3/3$ floating point operations 
\cite[\S 7.4.3]{GovL96}, while the second stage requires
$n^3/6$ floating point operations \cite[\S 6.57]{Wil65}
(or $4n^3/3$ flops if $A$ is symmetric \cite[\S 8.3.1]{GovL96}).
If the matrix $A$ is real, then only real arithmetic is needed --
in contrast to eigenvalue based methods which require complex
arithmetic if a real matrix has complex  eigenvalues.

\section{Symmetric Matrices}\label{s_sym}
In the first stage, La Budde's method reduces a real symmetric matrix
$A$ to tridiagonal form $T$
by orthogonal similarity transformations. The second stage, where it 
computes the coefficients of the characteristic polynomial of $T$,
amounts to the Sturm sequence method \cite{Giv53}.
We present recursions to compute individual coefficients in
the second stage of La Budde's method in \S \ref{s_ss}, describe
our assumptions for the floating point analysis in \S \ref{s_ass1},
and derive running error bounds in \S \ref{s_srun}.

\subsection{The Algorithm}\label{s_ss}
We present an implementation of the second
stage of La Budde's method for symmetric matrices.
Let
$$T=\begin{pmatrix}
\alpha_1&\beta_2& & &  \\
\beta_2& \alpha_2& \beta_3& &  \\
 &\ddots & \ddots& \ddots &\\
& &  \ddots&\ddots& \beta_n\\
  & & & \beta_n& \alpha_n
  \end{pmatrix}$$
be a $n\times n$ real symmetric tridiagonal matrix
with characteristic polynomial $p(\lambda)\equiv \det(\lambda I-T)$.
In the process of computing $p(\lambda)$,
the Sturm sequence method computes characteristic polynomials
$p_i(\lambda)\equiv\det(\lambda I-T_i)$ 
of all leading principal submatrices $T_i$ of order $i$,
where $p_n(\lambda)=p(\lambda)$.
The recursion for computing $p(\lambda)$ 
is \cite{Giv53}, \cite[(8.5.2)]{GovL96}
\begin{eqnarray}\label{e_sturm}
p_0(\lambda)&=&1, \quad
p_1(\lambda)=\lambda-\alpha_1\nonumber\\
p_i(\lambda)&=&(\lambda-\alpha_i)p_{i-1}(\lambda)-\beta_i^2 p_{i-2}(\lambda),
\qquad 2\leq i\leq n.
\end{eqnarray}
In order to recover individual coefficients of $p(\lambda)$
from the recursion (\ref{e_sturm}), we identify the polynomial 
coefficients
$$p(\lambda)=\lambda^n+c_1\lambda^{n-1}+\cdots+c_{n-1}\lambda+c_n$$
and
$$p_i{(\lambda)}=
\lambda^i+c_1^{(i)}\lambda^{i-1}+\cdots+c_{i-1}^{(i)}\lambda+c_i^{(i)},
\qquad 1\leq i\leq n,$$
where $c_k^{(n)}=c_k$.
Equating like powers of $\lambda$ on both sides of (\ref{e_sturm}) gives
recursions for individual coefficients $c_k$, which are 
presented as Algorithm 1. 
In the process, $c_1,\ldots,c_{k-1}$ are also computed.

\begin{algorithm}  
\caption{La Budde's method for symmetric tridiagonal matrices}
\algsetup{indent=3em}
\begin{algorithmic}[2]
\REQUIRE $n\times n$ real symmetric tridiagonal matrix $T$, index $k$ 
\ENSURE  Coefficient $c_1,\ldots, c_k$ of $p(\lambda)$
\STATE $c_1^{(1)} = -\alpha_1$

\STATE $c_1^{(2)} = c_1^{(1)}-\alpha_2$,  $c_2^{(2)} = \alpha_1\alpha_2-\beta_2^2$
\FOR {$i=3:k$}
\STATE $c_1^{(i)}=c_1^{(i-1)}-\alpha_i$
\STATE $c_2^{(i)} = c_2^{(i-1)} -\alpha_ic_1^{(i-1)}-\beta_i^2$
\FOR {$j=3:i-1$}
\STATE $c_j^{(i)}=c_j^{(i-1)}-\alpha_ic_{j-1}^{(i-1)}-\beta_i^2 c_{j-2}^{(i-2)}$
\ENDFOR
\STATE $c_i^{(i)}=-\alpha_ic_{i-1}^{(i-1)}-\beta_i^2 c_{i-2}^{(i-2)}$
\ENDFOR
\FOR {$i=k+1:n$}
\STATE $c_1^{(i)}=c_1^{(i-1)}-\alpha_i$
\IF {$k\geq 2$}
\STATE $c_2^{(i)} = c_2^{(i-1)} -\alpha_ic_1^{(i-1)}-\beta_i^2$
\FOR {$j=3:k$}
\STATE $c_j^{(i)}=c_j^{(i-1)}-\alpha_ic_{j-1}^{(i-1)}-\beta_i^2 c_{j-2}^{(i-2)}$
\ENDFOR
\ENDIF
\ENDFOR
\STATE \COMMENT
{Now $c_j=c_j^{(n)}$, $1\leq j\leq k$}
\end{algorithmic}
\end{algorithm}

If $T$ is a diagonal matrix then Algorithm 1 reduces to the Summation 
Algorithm \cite[Algorithm 1]{ReI10a} for computing characteristic polynomials
from eigenvalues. The Summation Algorithm is the basis for
MATLAB's \texttt{poly} function, which applies it to eigenvalues 
computed by \texttt{eig}.
The example in Figure \ref{f_ex} shows the coefficients computed by Algorithm 1
when $n=5$ and $k=3$.

\begin{figure}
{\small
$$\begin{array}{l||l|lr|lr}
i & c_1^{(i)} & c_2^{(i)} && c_3^{(i)}&\\
\hline\hline
1 & c_1^{(1)}= -\alpha_1&&\\
2 & c_1^{(2)} = c_1^{(1)}-\alpha_2 &c_2^{(2)}=&\alpha_1\alpha_2-\beta_2^2\\
3 & c_1^{(3)} = c_1^{(2)}-\alpha_3 &
c_2^{(3)}=&c_2^{(2)}-\alpha_3c_1^{(2)}-\beta_3^2& 
c_3^{(3)}=&-\alpha_3c_2^{(2)}-\beta_3^2c_1^{(1)}\\
4 & c_1^{(4)} = c_1^{(3)}-\alpha_4 &
c_2^{(4)}=&c_2^{(3)}-\alpha_4c_1^{(3)}-\beta_4^2& 
c_3^{(4)}=&c_3^{(3)}-\alpha_4c_2^{(3)}-\beta_4^2c_1^{(2)}\\
5 & c_1^{(5)} = c_1^{(4)}-\alpha_5 &
c_2^{(5)}=&c_2^{(4)}-\alpha_5c_1^{(4)}-\beta_5^2& 
c_3^{(5)}=&c_3^{(4)}-\alpha_5c_2^{(4)}-\beta_5^2c_1^{(3)}
\end{array}$$}
\caption{Coefficients computed by Algorithm 1 for $n=5$ and $k=3$.}\label{f_ex}
\end{figure}

\newpage
\subsection{Assumptions for Running Error Bounds}\label{s_ass1}
We assume that all matrices are real.  Error bounds for complex
matrices are derived in \cite[\S 6]{Rizth}. In addition, we make the
following assumptions:

\begin{enumerate}
\item The matrix elements are normalized real floating point numbers.
\item The coefficients computed in floating point arithmetic
are denoted by $\hat{c}_k^{(i)}$.  
\item The output from the floating point 
computation of Algorithms 1 and 2 is $\fl[c_k]\equiv\hat{c}_k^{(n)}$.
In particular, $\fl[c_1]\equiv c_1=-\alpha_1$.
\item The error in the computed coefficients is $e_k^{(i)}$ so that
$e_1^{(1)}=0$ and  
\begin{eqnarray}\label{e_error}
\hat c_k^{(i)}=c_k^{(i)}+e_k^{(i)}, \qquad 2\leq i\leq n,\quad 1\leq k\leq n.
\end{eqnarray}
\item The operations  do not cause underflow or overflow.
\item The symbol $u$ denotes the unit roundoff, and $nu<1$.
\item Standard error model for real floating point 
arithmetic  \cite[\S 2.2]{Hig02}:

If $\op\in \{+,-,\times,/\}$, and
$x$ and $y$ are real normalized floating point numbers
so that $x \op y$ does not underflow or overflow, then
\begin{equation} \label{model1}
\fl[x  \op  y]=(x \op  y)(1+\delta) \qquad \mathrm{where}
\qquad |\delta| \>\leq u,
\end{equation}
and
\begin{equation} \label{model2}
\fl[x \op  y]=\frac{x \op y}{1+\epsilon} \qquad \mathrm{where}
 \qquad |\epsilon|\leq u.
\end{equation}
\end{enumerate}

The following relations are required for the error bounds.

\begin{lemma}[Lemma 3.1 and Lemma 3.3 in  \cite{Hig02}]\label{theta}
Let $\delta_i$ and $\rho_i$ be real numbers, $1\leq i\leq n$, with 
$|\delta_i|\leq u$ and $\rho_i=\pm1$. If $nu<1$ then
\begin{enumerate}
\item $\prod_{i=1}^{n}(1+\delta_i)^{\rho_i}=1+\theta_n$,
where
$$|\theta_n|\leq \gamma_n\equiv \frac{n u}{1-n u}.$$
\item $(1+\theta_j)(1+\theta_k)=1+\theta_{j+k}$
\end{enumerate}
\end{lemma}

\subsection{Running Error Bounds}\label{s_srun}
We derive running error bounds for Algorithm 1,
first for $\hat{c}_1$, then for $\hat{c}_2$, and at last 
for the remaining coefficients $\hat{c}_j$, $3\leq j\leq k$.

The bounds below apply to lines 2, 4, and 14 of Algorithm 1.

\begin{theorem}[Error bounds for $\hat{c}_1^{(i)}$]\label{error_1}
If the assumptions in \S \ref{s_ass1} hold and 
$$\hat{c}_1^{(i)}=\fl\left[\hat{c}_1^{(i-1)}-\alpha_i\right],
\qquad 2\leq i\leq n,$$
then
$$|e_1^{(i)}|\leq |e_1^{(i-1)}|+u\>|\hat c_1^{(i)}|, \qquad 2\leq i\leq n.$$
\end{theorem}

\begin{proof}
The model (\ref{model2}) implies 
$(1+\epsilon^{(i)}) \hat c_1^{(i)}= \hat{c}_1^{(i-1)}-\alpha_i$
where $|\epsilon^{(i)}|\leq u$.
Writing the computed coefficients $\hat c_1^{(i)}$  and $\hat c_1^{(i-1)}$ 
in terms of their errors (\ref{e_error}) and 
then simplifying gives
$$e_1^{(i)}=e_1^{(i-1)}-\epsilon^{(i)}\hat c_1^{(i)}.$$
Hence $|e_1^{(i)}|\leq |e_1^{(i-1)}|+u\>|\hat c_1^{(i)}|$.
\end{proof}

The bounds below apply to lines 2, 6, and 16 of Algorithm 1.

\begin{theorem}[Error bounds for $\hat{c}_2^{(i)}$]\label{error_2}
If the assumptions in \S \ref{s_ass1} hold, and
\begin{eqnarray*}
\hat{c}_2^{(2)} & = &\fl\biggl[\fl\left[\alpha_1\alpha_2\right]-
\fl\left[\beta_2^2\right]\biggr]\\
\hat c_2^{(i)} & = &\fl\biggl[\fl\left[\hat c_2^{(i-1)}-
\fl\left[\alpha_i\hat c_1^{(i-1)}\right]\right]-\fl\left[\beta_i^2\right]\biggr],
\qquad 3\leq i\leq n,
\end{eqnarray*}
then
\begin{eqnarray*}
|e_2^{(2)}|& \leq & u\>\left(|\alpha_2\alpha_1|+|\beta_2^2|+
|\hat c_2^{(2)}|\right)\\
|e_2^{(i)}|&\leq &|e_2^{(i-1)}|+|\alpha_ie_1^{(i-1)}|+
u\>\left(|\hat c_2^{(i-1)}|+|\beta_i^2|+|\hat c_2^{(i)}|\right)+
\gamma_2\>|\alpha_i\hat{c}_1^{(i-1)}|.
\end{eqnarray*}
\end{theorem}

\begin{proof}
The model (\ref{model1}) implies for the multiplications
$$\hat c_2^{(2)}=\fl\biggl[\alpha_2\alpha_1(1+\delta)-\beta_2^2(1+\eta)\biggr],
\qquad \mathrm{where}\quad |\delta|,|\eta|\leq u.$$
Applying the model (\ref{model2}) to the subtraction gives
$$(1+\epsilon)\hat c_2^{(2)}=\alpha_2\alpha_1(1+\delta)-\beta_2^2(1+\eta), 
\qquad |\epsilon|\leq u.$$
Now express  $\hat c_2^{(2)}$ in terms of the errors $e_2^{(2)}$ 
from (\ref{e_error}) and simplify.

For $3\leq i\leq n$, applying model (\ref{model1}) to the multiplications
and the first subtraction gives
$\hat c_2^{(i)}=\fl\left[g_1^{(i)}-g_2^{(i)}\right]$,
where 
$$g_1^{(i)}\equiv\hat c_2^{(i-1)}(1+\delta^{(i)})-
\alpha_i\hat{c}_1^{(i-1)}(1+\theta_{2}^{(i)}),\qquad
g_2^{(i)}=\beta_i^2(1+\eta^{(i)}),$$
$|\delta^{(i)}|$, $|\eta^{(i)}|\leq u$ and 
$|\theta_{2}^{(i)}| \leq \gamma_2$.
Applying the model (\ref{model2}) to the remaining subtraction gives
$$(1+\epsilon^{(i)})\hat  c_2^{(i)}=
\hat{c}_2^{(i-1)}(1+\delta^{(i)})-\alpha_i \hat {c}_1^{(i-1)}(1+\theta_{2}^{(i)})-
\beta_i^2(1+\eta^{(i)}),$$
where $|\epsilon^{(i)}|\leq u$. Now express $\hat{c}_2^{(i-1)}$ and 
$\hat{c}_1^{(i-1)}$ in terms of their errors (\ref{e_error}) to get
$$e_2^{(i)}= e_2^{(i-1)}+\delta^{(i)} \hat c_2^{(i-1)}-\alpha_i e_1^{(i-1)}-
\theta_{2}^{(i)}\alpha_i\hat c_1^{(i-1)}-\beta_i^2\eta^{(i)}-\epsilon^{(i)} \hat c_2^{(i)},$$
and apply the  triangle inequality.
\end{proof}

The bounds below apply to lines 8, 10, and 18 of Algorithm 1.

\begin{theorem}[Error bounds for $\hat{c}_j^{(i)}$,
$3\leq j\leq k$]\label{error_3}
If the assumptions in \S \ref{s_ass1} hold, and
\begin{eqnarray*}
\hat c_i^{(i)}&=&-\fl\biggl[\fl\left[\alpha_i\hat c_{i-1}^{(i-1)}\right]+
\fl\left[\beta_i^2\hat c_{i-2}^{(i-2)}\right]\biggr],\qquad 3\leq i\leq k,\\
\hat c_j^{(i)}&=&\fl\biggl[\fl\left[\hat{c}_j^{(i-1)}-
\fl\left[\alpha_i\hat c_{j-1}^{(i-1)}\right]\right]-
\fl\left[\beta_i^2\hat c_{j-2}^{(i-2)}\right]\biggr],
\qquad 3\leq j\leq k, \quad j+1\leq i\leq n,
\end{eqnarray*}
then
\begin{eqnarray*}
|e_i^{(i)}|&\leq &|\alpha_i e_{i-1}^{(i-1)}|+|\beta_i^2e_{i-2}^{(i-2)}|+ 
u\>\left(|\alpha_i \hat c_{i-1}^{(i-1)}|+|\hat c_i^{(i)}|\right)
+\gamma_2\>|\beta_i^2\>\hat c_{i-2}^{(i-2)}|\\
|e_j^{(i)}|&\leq &|e_j^{(i-1)}|+|\alpha_i e_{j-1}^{(i-1)}|+
|\beta_i^2 e_{j-2}^{(i-2)}|\\
&&+u\>\left(|\hat c_{j}^{(i-1)}|+|\hat c_j^{(i)}|\right)+
\gamma_2\>\left(|\alpha_i\hat c_{j-1}^{(i-1)}|+
|\beta_i^2 \> \hat c_{j-2}^{(i-2)}|\right).
\end{eqnarray*}
\end{theorem}

\begin{proof}
The model (\ref{model1}) implies for the three multiplications that
$$\hat c_i^{(i)}=-\fl \biggl[\alpha_i\hat c_{i-1}^{(i-1)}(1+\delta)+
\beta_i^2\hat c_{i-2}^{(i-2)}(1+\theta_{2})\biggr],$$
where $|\delta|\leq u$ and $|\theta_{2}|\leq \gamma_2$. 
Applying model (\ref{model2}) to the remaining addition gives
$$(1+\epsilon)\hat c_i^{(i)}=
-\alpha_i\hat{c}_{i-1}^{(i-1)}(1+\delta)-\beta_i^2\hat{c}_{i-2}^{(i-2)}
(1+\theta_{2}), \qquad |\epsilon|\leq u.$$
As in the previous proofs, write  
$\hat{c}_i^{(i)}$, $\hat c_{i-1}^{(i-1)}$ and $\hat c_{i-2}^{(i-2)}$ in terms of
their errors (\ref{e_error}), 
$$e_i^{(i)}=-\alpha_ie_{i-1}^{(i-1)}-\beta_i^2 e_{i-2}^{(i-2)}-
\theta_2 \beta_i^2\hat c_{i-2}^{(i-2)}-\delta \alpha_i\hat c_{i-1}^{(i-1)}-
\epsilon \hat c_i^{(i)}.$$
and apply the triangle inequality. 

For $k+1\leq  i\leq n$, applying (\ref{model2}) to the two 
multiplications and the first subtraction gives
$c_j^{(i)}=\fl\left[g_1^{(i)}-g_2^{(i)}\right]$ where 
$$g_1^{(i)}\equiv \hat{c}_j^{(i-1)}(1+\delta^{(i)})-
\alpha_i \hat{c}_{j-1}^{(i-1)}(1+\theta_{2}^{(i)}),\qquad
g_2^{(i)}\equiv \beta_i^2\hat c_{k-2}^{(i-2)}(1+\hat \theta^{(i)}_{2}),$$
$|\delta^{(i)}|\leq u$ and $|\theta_{2}^{(i)}|$,
$|\hat{\theta}^{(i)}_{2}|\leq \gamma_2$.  
Applying model (\ref{model2}) to the remaining subtraction gives
$$(1+\epsilon^{(i)})\hat c_j^{(i)}=
\hat{c}_j^{(i-1)}(1+\delta^{(i)})-\alpha_i \hat c_{j-1}^{(i-1)}(1+\theta_{2}^{(i)})-
\beta_i^2\hat c_{j-2}^{(i-2)}(1+\hat \theta^{(i)}_{2}),$$
where $|\epsilon^{(i)}|\leq u$.
Write the computed coefficients in terms of their errors (\ref{e_error}),
\begin{eqnarray*}
e_j^{(i)}&=&e_j^{(i-1)}+\delta^{(i)}\hat c_j^{(i-1)}-
\theta_{2}^{(i)} \alpha_i \hat{c}_{j-1}^{(i-1)}-\alpha_i e_{j-1}^{(i-1)}-
\beta_i^2 e_{j-2}^{(i-2)}
-\hat{\theta}^{(i)}_{2} \beta_i^2 \hat c_{j-2}^{(i-2)}-
\epsilon^{(i)}\hat c_j^{(i)},
\end{eqnarray*}
and apply the triangle inequality.
\end{proof}

We state the bounds when the leading $k$ coefficients of $p(\lambda)$
are computed by Algorithm 1 in floating point arithmetic.

\begin{corollary}[Error Bounds for $\fl(c_j)$,
 $1\leq j\leq k$]\label{c_srunerr}
If the assumptions in \S \ref{s_ass1} hold, then
$$|\fl[c_j]-c_j|\leq \phi_j,\qquad 1\leq j\leq k,$$
where $\fl[c_j]\equiv \hat{c}_j^{(n)}$ and
$\phi_j\equiv |e_j^{(n)}|$ are given in Theorems \ref{error_1},
\ref{error_2} and \ref{error_3}.
\end{corollary}

\section{Nonsymmetric Matrices}\label{s_nonsym}
In the first stage, La Budde's method \cite{Giv57}
reduces a real square matrix to upper Hessenberg form $H$.
In the second stage it computes the coefficients
of the characteristic polynomial of $H$.
We present recursions to compute individual coefficients in 
the second stage of La Budde's method
in \S \ref{s_lab}, and derive running error bounds in \S \ref{s_labrun}.

\subsection{The Algorithm}\label{s_lab}
We present an implementation of the second stage of La Budde's method
for nonsymmetric matrices.
Let 
$$H=\begin{pmatrix}\alpha_1 & h_{12} & \ldots & \ldots & h_{1n}\\
                   \beta_2& \alpha_2& h_{23}& &\vdots\\
                          &\ddots &\ddots& \ddots&\vdots  \\
                           & & \ddots &\ddots &h_{n-1,n}\\
                          & & & \beta_n & \alpha_n
\end{pmatrix}$$
be a real $n\times n$ upper Hessenberg matrix with 
diagonal elements $\alpha_i$, subdiagonal elements $\beta_i$,
and characteristic polynomial $p(\lambda)\equiv\det(\lambda I -H)$.

La Budde's method computes the characteristic
polynomial of an upper Hessenberg matrix $H$ by successively computing
characteristic polynomials of 
leading principal submatrices $H_i$ of order $i$ \cite{Giv57}. 
Denote the characteristic polynomial of
$H_i$ by $p_i(\lambda)=\det(\lambda I-H_i)$, $1\leq i\leq n$,
where $p(\lambda)=p_n(\lambda)$.
The recursion for computing $p(\lambda)$ is \cite[(6.57.1)]{Wil65}
\begin{eqnarray}\label{e_uh}
p_0(\lambda)&=&1, \quad p_1(\lambda)=\lambda-\alpha_1\nonumber\\
p_i(\lambda) &= &(\lambda-\alpha_i)p_{i-1}(\lambda)-
\sum_{m=1}^{i-1}{h_{i-m,i}\>\beta_i\cdots \beta_{i-m+1}\>p_{i-m-1}(\lambda)},
\end{eqnarray}
where $2\leq i\leq n$.
The recursion for $p_i(\lambda)$ is obtained by developing the determinant
of $\lambda I-H_i$ along the last row of $H_i$. Each term in the
sum contains an element in the last column of $H_i$ and a product of 
subdiagonal elements.

As in the symmetric case, we let
$$p(\lambda)=\lambda^n+c_1\lambda^{n-1}+\cdots+c_{n-1}\lambda+c_n$$
and
$$p_i{(\lambda)}=
\lambda^i+c_1^{(i)}\lambda^{i-1}+\cdots+c_{i-1}^{(i)}\lambda+c_i^{(i)},
\qquad 1\leq i\leq n,$$
where $c_k^{(n)}=c_k$. Equating like powers of $\lambda$ in (\ref{e_uh})
gives recursions for individual coefficients $c_k$, which are
presented as Algorithm 2.
In the process, $c_1,\ldots,c_{k-1}$ are also computed.

\begin{algorithm}  
\caption{La Budde's method for upper Hessenberg matrices}
\algsetup{indent=1em}
\begin{algorithmic}[1]
\REQUIRE  $n \times n$ real upper Hessenberg matrix $H$, index $k$  
\ENSURE    Coefficient $c_k$ of $p(\lambda)$ 
\STATE $c_1^{(1)}=-\alpha_1$
\STATE $c_1^{(2)}=c_1^{(1)}-\alpha_2$, 
$c_{2}^{(2)}=\alpha_1\alpha_2- h_{12}\beta_2$
\FOR {$i=3:k$}
\STATE $c_1^{(i)}=c_1^{(i-1)}-\alpha_i$
\FOR {$j=2:i-1$}
\STATE \mbox{\small $c_j^{(i)}=c_j^{(i-1)}-\alpha_ic_{j-1}^{(i-1)}-
\sum_{m=1}^{j-2}{h_{i-m,i}\>\beta_i\cdots \beta_{i-m+1}\>c_{j-m-1}^{(i-m-1)}} - 
h_{i-j+1,i}\>\beta_i\cdots\beta_{i-j+2}$}
\ENDFOR
\STATE \mbox{$c_i^{(i)}=-\alpha_ic_{i-1}^{(i-1)}-
\sum_{m=1}^{i-2}{h_{i-m,i}\>\beta_i\cdots \beta_{i-m+1}\>c_{i-m-1}^{(i-m-1)}} - 
h_{1i}\>\beta_i\cdots\beta_{2}$}
\ENDFOR
\FOR {$i=k+1:n$}
\STATE $c_1^{(i)}=c_1^{(i-1)}-\alpha_i$
\IF{$k\geq 2$}
\FOR {$j=2:k$}
\STATE \mbox{\small $c_j^{(i)}=c_j^{(i-1)}-\alpha_ic_{j-1}^{(i-1)}-
\sum_{m=1}^{j-2}{h_{i-m,i}\>\beta_i\cdots \beta_{i-m+1}\>c_{j-m-1}^{(i-m-1)}} - 
h_{i-j+1,i}\>\beta_i\cdots\beta_{i-j+2}$}
\ENDFOR
\ENDIF
\ENDFOR
\STATE \COMMENT{Now  $c_j=c_j^{(n)}$, $1\leq j\leq k$}
\end{algorithmic}
\end{algorithm} 

For the special case when $H$ is symmetric and tridiagonal, Algorithm 2 
reduces to Algorithm 1. Figure \ref{f_ex1} shows an example of
the recursions for $n=5$ and $k=3$.

\begin{figure}
{\small
$$\begin{array}{l||l|lr|lr}
i & c_1^{(i)} & c_2^{(i)} && c_3^{(i)}&\\
\hline\hline
1 & c_1^{(1)}= -\alpha_1&&\\
2 & c_1^{(2)} = c_1^{(1)}-\alpha_2 &c_2^{(2)}=&\alpha_1\alpha_2-h_{12}\beta_2\\
3 & c_1^{(3)} = c_1^{(2)}-\alpha_3 &
c_2^{(3)}=&c_2^{(2)}-\alpha_3c_1^{(2)}-h_{23}\beta_3& 
c_3^{(3)}=&-\alpha_3c_2^{(2)}-h_{23}\beta_3c_1^{(1)}-h_{13}\beta_3\beta_2\\
4 & c_1^{(4)} = c_1^{(3)}-\alpha_4 &
c_2^{(4)}=&c_2^{(3)}-\alpha_4c_1^{(3)}-h_{34}\beta_4& 
c_3^{(4)}=&c_3^{(3)}-\alpha_4c_2^{(3)}-h_{34}\beta_4c_1^{(2)}
-h_{24}\beta_4\beta_3\\
5 & c_1^{(5)} = c_1^{(4)}-\alpha_5 &
c_2^{(5)}=&c_2^{(4)}-\alpha_5c_1^{(4)}-h_{45}\beta_5& 
c_3^{(5)}=&c_3^{(4)}-\alpha_5c_2^{(4)}-h_{45}\beta_5c_1^{(3)}
-h_{35}\beta_5\beta_4
\end{array}$$
}
\caption{Coefficients Computed by Algorithm 2 when $n=5$ and 
$k=3$.}\label{f_ex1}
\end{figure}

Algorithm 2 computes the characteristic polynomial of companion matrices
exactly. To see this, consider the $n\times n$ companion matrix of the form
$$ \begin{pmatrix}0&       & & -c_n\\
                    1&\ddots&  &\vdots     \\
                     &\ddots&0  & -c_2\\
                     &      &1 &-c_1\end{pmatrix}.$$ 
Algorithm 2 computes 
$c_j^{(i)}=0$ for $1\leq j\leq n$ and $1\leq i\leq n-1$,
so that $c_j^{(n)}=c_j$. Since only trivial arithmetic operations are 
performed, Algorithm 2 computes
the characteristic polynomial exactly.

\subsection{Running Error Bounds}\label{s_labrun}
We present running error bounds for the coefficients of $p(\lambda)$
of a real Hessenberg matrix $H$. 

The bounds below apply to lines 2, 4, and 11 of Algorithm 2.

\begin{theorem}[Error bounds for $\hat{c}_1^{(i)}$]\label{t_1}
If the assumptions in \S \ref{s_ass1} hold and 
$$\hat{c}_1^{(i)}=\fl\left[\hat{c}_1^{(i-1)}-\alpha_i\right],
\qquad 2\leq i\leq n,$$
then
$$|e_1^{(i)}|\leq |e_1^{(i-1)}|+u\>|\hat c_1^{(i)}|, \qquad 2\leq i\leq n.$$
\end{theorem}

\begin{proof}
The proof is the same as that of Theorem \ref{error_1}.
\end{proof}

The following bounds apply to line 2 of Algorithm 2,
as well as lines 6 and 14 for the case $j=2$.

\begin{theorem}[Error bounds for $\hat{c}_2^{(i)}$]\label{t_2}
If the assumptions in \S \ref{s_ass1} hold, and
\begin{eqnarray*}
\hat{c}_2^{(2)}&=&\fl\biggl[\fl\left[\alpha_1\alpha_2\right]-
\fl\left[h_{12}\>\beta_2\right]\biggr]\\
\hat c_2^{(i)}&=&\fl\biggl[\fl\left[\hat c_2^{(i-1)}-
\fl\left[\alpha_i\hat{c}_1^{(i-1)}\right]\right]-
\fl\left[h_{i-1,i}\>\beta_i\right]\biggr], \qquad 3\leq i\leq n,
\end{eqnarray*}
then
\begin{eqnarray*}
|e_2^{(2)}|& \leq &u\>\left(|\alpha_2\alpha_1|+|h_{12}\beta_2|+
|\hat c_2^{(2)}|\right)\\
|e_2^{(i)}|&\leq &|e_2^{(i-1)}|+|\alpha_ie_1^{(i-1)}|+
u\>\left(|\hat c_2^{(i-1)}|+|h_{i-1,i}\>\beta_i|+|\hat c_2^{(i)}|\right)+
\gamma_2\>|\alpha_i\hat{c}_1^{(i-1)}|.
\end{eqnarray*}
\end{theorem}

\begin{proof}
The proof is the same as that of Theorem \ref{error_2}.
\end{proof}

The bounds below apply to lines 6, 8, and 14 of Algorithm 2.

\begin{theorem}[Error bounds for $\hat{c}_j^{(i)}$, $3\leq j\leq k$]\label{t_3}
If the assumptions in \S \ref{s_ass1} hold, 
\begin{eqnarray*}
 \hat c_i^{(i)}=-\fl\left[\fl\left[\alpha_i\hat c_{i-1}^{(i-1)}\right]+
\fl\left[\sum_{m=1}^{i-2}{h_{i-m,i}\>\beta_i\cdots \beta_{i-m+1}\>
\hat c_{i-m-1}^{(i-m-1)}}+h_{1i}\>\beta_i\cdots\beta_2\right]\right],
\end{eqnarray*}
$3\leq i\leq k$, and
\begin{eqnarray*}
\hat c_j^{(i)}&=&\fl\left[\fl\left[\hat c_j^{(i-1)}-
\fl\left[\alpha_i\hat c_{j-1}^{(i-1)}\right]\right]\right.\\
&& \qquad \left.  -
\fl\left[\sum_{m=1}^{j-2}{h_{i-m,i}\>\beta_i\cdots \beta_{i-m+1}\>
\hat c_{i-m-1}^{(i-m-1)}}-h_{i-j+1}\>\beta_i\cdots \beta_{i-j+2}\right]\right],
\end{eqnarray*}
$3\leq j\leq k$, $j+1\leq i\leq n$, then
\begin{eqnarray*}
|e_i^{(i)}|&\leq & |\alpha_ie_{i-1}^{(i-1)}|+
\sum_{m=1}^{i-2}{|h_{i-m,i}\>\beta_i\cdots \beta_{i-m+1}\>e_{i-m-1}^{(i-m-1)}|}\\
&& + \gamma_{i+1}\>\sum_{m=2}^{i-2}{|h_{i-m,i}\>\beta_i\cdots \beta_{i-m+1}\>
\hat c_{i-m-1}^{(i-m-1)}|}\\
&& + \gamma_{i}\>\left(|h_{i-1,i}\>\beta_i\>\hat c_{i-2}^{(i-2)}|+
|h_{1i}\>\beta_i\cdots \beta_2|\right)
+u\> \left(|\hat c_i^{(i)}|+|\alpha_i\hat c_{i-1}^{(i-1)}|\right)\\
\end{eqnarray*}
and 
\begin{eqnarray*}
|e_j^{(i)}|&\leq &|e_j^{(i-1)}|+|\alpha_ie_{j-1}^{(i-1)}|+
\sum_{m=1}^{j-2}{|h_{i-m,i}\>\beta_i\cdots \beta_{i-m+1}\> e_{j-m-1}^{(i-m-1)}|}\\
&&+\gamma_{j+1}\>\left(\sum_{m=2}^{j-2}{|h_{i-m,i}\>\beta_i\cdots \beta_{i-m+1}\>
\hat{c}_{j-m-1}^{(i-m-1)}|}\right)\\
&&+\gamma_{j}\>\left(|h_{i-1,i}\>\beta_i \>\hat c_{j-2}^{(i-2)}|+
|h_{i-j+1,i}\>\beta_i\cdots \beta_{i-j+2}|\right)\\
&&+u \> \left(|\hat c_j^{(i)}|+|\hat c_j^{(i-1)}|\right)
+\gamma_{2} \>|\alpha_i\hat c_{j-1}^{(i-1)}|.
\end{eqnarray*}
\end{theorem}

\begin{proof}
The big sum in $\hat c_i^{(i)}$ contains $i-2$ products, where
each product consists of $m+2$ numbers.  For the $m+1$ multiplications in 
such a product, the model (\ref{model1}) and Lemma \ref{theta} imply
\begin{eqnarray*}
g_m\equiv\fl\left[h_{i-m,i}\>\beta_i\cdots \beta_{i-m+1}\>
\hat c_{i-m-1}^{(i-m-1)}\right]
=h_{i-m,i}\>\beta_i\cdots \beta_{i-m+1}\>\hat c_{i-m-1}^{(i-m-1)}(1+\theta_{m+1}), 
\end{eqnarray*}
where $|\theta_{m+1}|\leq \gamma_{m+1}$ and $1\leq m\leq i-2$.
The term $h_{1i}\>\beta_i\cdots\beta_2$ is a product of $i$ numbers, so that
$$g_{i-1}\equiv\fl\left[h_{1i}\>\beta_i\cdots \beta_2\right]
=h_{1i}\>\beta_i\cdots \beta_2(1+\theta_{i-1}),$$
where $|\theta_{i-1}|\leq \gamma_{i-1}$.
Adding the $i-1$ products $g_m$ from left to right, so that
$$g\equiv\fl\left[ \ldots \fl\left[ \fl\left[g_1+g_2\right]+g_3\right]
\cdots +g_{i-1}\right],$$
gives, again with (\ref{model1}), the relation
\begin{eqnarray*}
&&g=h_{i-1,i}\>\beta_i\>\hat c_{i-2}^{(i-2)}(1+\theta_i)
+\sum_{m=2}^{i-2}{h_{i-m,i}\>\beta_i\cdots \beta_{i-m+1}\>
\hat c_{i-m-1}^{(i-m-1)}}\left(1+\theta_{i+1}^{(m)}\right)\\
&& \qquad +h_{1i}\>\beta_i\cdots \beta_2\>\left(1+\hat \theta_{i}\right),
\end{eqnarray*}
where $|\theta_{i+1}^{(m)}|\leq \gamma_{i+1}$ and  
$|\theta_i|, |\hat \theta_i|\leq \gamma_i$. 
For the very first term in $\hat{c}_i^{(i)}$ we get
$\fl\left[\alpha_i\hat c_{i-1}^{(i-1)}\right]=
\alpha_i\hat c_{i-1}^{(i-1)}(1+\delta)$, where $|\delta|\leq u$.
Adding this term to $g$ and using model (\ref{model2}) yields
\begin{eqnarray*}
-(1+\epsilon)\hat c_i^{(i)}&=&\alpha_i\hat c_{i-1}^{(i-1)}(1+\delta)+
h_{i-1,i}\>\beta_i\>\hat c_{i-2}^{(i-2)}(1+\theta_i)+
h_{1i}\>\beta_i\cdots \beta_2\>(1+\hat \theta_i)\\
&&+\sum_{m=2}^{i-2}{h_{i-m,i}\beta_i\cdots \beta_{i-m+1}\>
\hat c_{i-m-1}^{(i-m-1)}\left(1+\theta_{i+1}^{(m)}\right)},
\end{eqnarray*}
where $|\epsilon|\leq u$. 
Write the computed coefficients in terms of their errors (\ref{e_error})
\begin{eqnarray*}
-e_i^{(i)}&=&\alpha_ie_{i-1}^{(i-1)}+
\sum_{m=1}^{i-2}{h_{i-m,k}\>\beta_i\cdots \beta_{i-m+1}\>
e_{i-m-1}^{(i-m-1)}}+\alpha_i\hat c_{i-1}^{(i-1)}\>\delta\\
&&+h_{i-1,i}\>\beta_i\>\hat c_{i-2}^{(i-2)}\theta_i
+\sum_{m=2}^{i-2}{h_{i-m,i}\>\beta_i\cdots \beta_{i-m+1}
\>\hat c_{i-m-1}^{(i-m-1)}\>\theta_{i+1}^{(m)}}\\
&&+h_{1i}\>\beta_i\cdots \beta_2\>\hat\theta_i
+\epsilon\>\hat c_i^{(i)},
\end{eqnarray*}
and then apply the triangle inequality.

For $j+1\leq i$, $\hat{c}_j^{(i)}$ contains the additional term 
$\hat c_j^{(i-1)}$, which is involved in the first subtraction.
Model (\ref{model1}) implies
$$\fl\left[\hat c_j^{(i-1)}-\fl\left[\alpha_i\hat c_{j-1}^{(i-1)}\right]
\right]=\hat c_j^{(i-1)}\>\left(1+\delta^{(i)}\right)-
\alpha_i\hat c_{j-1}^{(i-1)}\left(1+\theta_2^{(i)}\right),$$
where $|\delta^{(i)}|\leq u$ and $|\theta_2^{(i)}|\leq \gamma_2$. 
From this we subtract $g$ which is computed as in the case $j=i$.
\end{proof} 

Finally we can state bounds when the leading $k$ coefficients of $p(\lambda)$
are computed by Algorithm 2 in floating point arithmetic.

\begin{corollary}[Error Bounds for $\fl(c_j)$, $1\leq j\leq k$]\label{c_runerr}
If the assumptions in \S \ref{s_ass1} hold, then
$$|\fl[c_j]-c_j|\leq \rho_j,\qquad 1\leq j\leq k,$$
where $\fl[c_j]\equiv \hat c_j^{(n)}$ and
$\rho_j\equiv |e_j^{(n)}|$ are given in Theorems \ref{t_1},
\ref{t_2} and \ref{t_3}.
\end{corollary}

\paragraph{Potential instability of La Budde's method}
The running error bounds reflect the potential instability of La
Budde's method.  The coefficient $c_j^{(i)}$ is computed from
the preceding coefficients $c_j^{(i-1)},\ldots,c_j^{(i-j+1)}$. La Budde's
 method can produce inaccurate results for $c_j^{(i)}$, if the
magnitudes of preceding coefficients are very large compared to
$c_j^{(i)}$ so that catastrophic cancellation occurs in the
computation of $c_j^{(i)}$. This means the error in the
computed coefficient $\hat c_j$ can be
large if the preceding coefficients in the characteristic polynomials
of the leading principal submatrices are larger than $\hat{c}_j$.

It may be that 
the instability of La Budde's method is related to the illconditioning
of the coefficients. Unfortunately we were not able to show this connection.

\section{Overall Error Bounds}\label{s_combined}
We present first order error bounds for both stages of La Budde's method.
The bounds take into the account the error from the 
reduction to Hessenberg (or tridiagonal) form in the first stage, 
as well as the roundoff error from the computation of the characteristic 
polynomial of the Hessenberg (or tridiagonal) matrix in the second stage.
We derive bounds for symmetric matrices in \S \ref{s_scombined},
and for nonsymmetric matrices in \S \ref{s_nscombined}.

\subsection{Symmetric Matrices}\label{s_scombined}
This bound combines the errors from the reduction of a
symmetric matrix $A$ to tridiagonal form $T$ with the roundoff
error from Algorithm 1.

Let $\tilde{T}=T+E$ be the tridiagonal
matrix computed in floating point arithmetic by applying
Householder similarity transformations to the symmetric matrix $A$.
From \cite[\S 8.3.1.]{GovL96} follows that for some small constant $\nu_1>0$ 
one can bound the error in the Frobenius norm by
\begin{eqnarray}\label{e_sym} 
\|E\|_F\leq  \>\nu_1 n^2\|A\|_F\>u.
\end{eqnarray}
The backward error $E$ can be
viewed as a matrix perturbation. This means we need to incorporate the
sensitivity of the coefficients $c_j$ to changes $E$ in the matrix.
The condition numbers that quantify this sensitivity 
can be expressed in terms
of elementary symmetric functions of the singular values
\cite{IpsR07}.
Let $\sigma_1\geq\ldots\geq \sigma_n$ be the singular values of $A$, 
and denote by 
$$s_0\equiv 1, \qquad
s_j\equiv \sum_{1\leq i_1<\cdots<i_j\leq n}{\sigma_{i_1}\cdots \sigma_{i_j}},
\qquad 1\leq j\leq n,$$
the $j$th elementary symmetric function in all $n$ singular values.

\begin{theorem}[Symmetric Matrices]\label{t_symerr}
If the assumptions in \S \ref{s_ass1} hold,
$A$ is real symmetric with
$\|A\|_F< 1/(\nu_1  n^2u)$ for the constant $\nu_1$ in (\ref{e_sym}), 
$\tilde{c}_j$ are the coefficients of the characteristic polynomial 
of $\tilde{T}$, then
$$|\fl[\tilde{c}_j]-c_j|\leq (n-j+1)\>s_{j-1}\> \nu_1 n^2\|A\|_F\>u+\phi_j+
\mathcal{O}\left(u^2\right),\qquad 1\leq j\leq k,$$
where $\phi_j$ are the running error bounds from Corollary \ref{c_srunerr}.
\end{theorem}

\begin{proof}
The triangle inequality implies
$$|\fl[\tilde{c}_j]-c_j|\leq |\fl[\tilde{c}_j]-\tilde c_j|+|\tilde c_j-c_j|.$$
Applying Corollary \ref{c_runerr} to the
first term gives $|\fl[\tilde{c}_j]-\tilde c_j|\leq \phi_j$.

Now we bound the second term $|\tilde c_j-c_j|$, and use the 
fact that $A$ and $T$ have the same singular values.
If $\|E\|_2<1$ then the absolute first order perturbation bound
\cite[Remark 3.6]{IpsR07} applied to $T$ and $T+E$ gives
$$|\tilde c_j-c_j|\leq  (n-j+1) s_{j-1}\>\|E\|_2\>u+
\mathcal{O}\left(\|E\|_2^2\right), \qquad 1\leq j\leq k.$$
From $(\ref{e_sym})$ follows $\|E\|_2\leq \|E\|_F\leq \nu_1 n^2 \|A\|_F\>u$.
Hence we need $\|A\|_F<1/(\nu_1 n^2 u)$ to apply the above perturbation bound.
\end{proof}

Theorem \ref{t_symerr} suggests two sources for the error in 
the computed coefficients $\fl[\tilde{c}_j]$: the sensitivity 
of $c_j$ to perturbations in the matrix, 
and the roundoff error $\rho_j$ introduced by Algorithm 1.
The sensitivity of $c_j$ to perturbations in the matrix is represented 
by the first order condition number $(n-j+1)s_{j-1}$, which amplifies
the error $\nu_1n^2\|A\|_F\>u$ from the reduction to tridiagonal form.

\subsection{Nonsymmetric Matrices}\label{s_nscombined}
This bound combines the errors from the reduction of a
nonsymmetric matrix $A$ to upper Hessenberg form $H$ with the roundoff
error from Algorithm 2.

Let $\tilde{H}=H+E$ be the upper
Hessenberg matrix computed in floating point arithmetic by applying
Householder similarity transformations to $A$.
From \cite[\S 7.4.3]{GovL96} follows that for some small constant $\nu_2>0$ 
\begin{eqnarray}\label{e_nonsym} 
\|E\|_F\leq  \>\nu_2 n^2\|A\|_F\>u.
\end{eqnarray}
The polynomial coefficients of nonsymmetric matrices are more sensitive
to changes in the matrix than those of symmetric matrices. 
The sensitivity is a function of only the largest singular values, rather
than all singular values \cite{IpsR07}. We define
$$s_0^{(1)}=1, \qquad
s_{j-1}^{(j)}\equiv \sum_{1\leq i_1<\cdots<i_{j-1}\leq j}{\sigma_{i_1}\cdots 
\sigma_{i_{j-1}}}\leq j\sigma_1\cdots\sigma_{j-1},
\qquad 1\leq j\leq n,$$
which is the $(j-1)$st elementary symmetric function in only the 
$j$ largest singular values.

\begin{theorem}[Nonsymmetric Matrices]\label{t_nonsymerr}
If the assumptions in \S \ref{s_ass1} hold, 
$\|A\|_F< 1/(\nu_2 n^2u)$ for the constant $\nu_2$ in (\ref{e_nonsym}),
and $\tilde{c}_j$ are the coefficients of the characteristic polynomial 
of $\tilde{H}$, then
$$\left|\fl[\tilde{c}_j]-c_j\right|\leq  
{n \choose j}s_{j-1}^{(j)}\> \nu_2 n^2\|A\|_F\>u+\rho_j
+\mathcal{O}\left(u^2\right),\qquad 1\leq j\leq k,$$
where $\rho_j$ are the running error bounds from Corollary \ref{c_runerr}.
\end{theorem}

\begin{proof}
The proof is similar to that of Theorem \ref{t_symerr}.
The triangle inequality implies
$$|\fl[\tilde{c}_j]-c_j|\leq |\fl[\tilde{c}_j]-\tilde c_j|+|\tilde c_j-c_j|.$$
Applying Corollary \ref{c_runerr} to the
first term gives $|\fl[\tilde{c}_j]-\tilde c_j|\leq \rho_j$.

Now we bound the second term $|\tilde c_j-c_j|$, and use the 
fact that $A$ and $H$ have the same singular values.
If $\|E\|_2<1$ then the absolute first order perturbation bound
\cite[Remark 3.4]{IpsR07} applied to $H$ and $H+E$ gives
$$|\tilde c_j-c_j|\leq  {n \choose j}s_{j-1}^{(j)}\|E\|_2\>u+
\mathcal{O}\left(\|E\|_2^2\right), \qquad 1\leq j\leq k.$$
From $(\ref{e_nonsym})$ follows $\|E\|_2\leq \|E\|_F\leq \nu_2 n^2 \|A\|_F\>u$.
Hence we need $\|A\|_F<1/(\nu_2 n^2 u)$ to apply the above perturbation bound.
\end{proof}

As in the symmetric case, there are two sources for the error in 
the computed coefficients $\fl[\tilde{c}_j]$: the sensitivity of $c_j$
to perturbations in the matrix, and 
the roundoff error $\rho_j$ introduced by Algorithm 2.
The sensitivity of $c_j$ to perturbations in the matrix is represented 
by the first order condition number ${n\choose j}s_{j-1}^{(j)}$, which amplifies
the error $\nu_2n^2\|A\|_F\>u$ from the reduction to Hessenberg form.

\section{Numerical Experiments}\label{s_exp}
We compare the accuracy of Algorithms 1 and 2 to MATLAB's \texttt{poly}
function, and demonstrate the performance of the running error bounds
from Corollaries \ref{c_srunerr} and \ref{c_runerr}.
The experiments illustrate that Algorithms 1 and 2 tend to be more accurate 
than \texttt{poly}, and sometimes substantially so,
especially when the matrices are indefinite or nonsymmetric.

We do not present plots for the overall error bounds in Theorems
\ref{t_symerr} and \ref{t_nonsymerr}, because they turned out to be
much more pessimistic than expected.  We conjecture that the errors
from the reduction to Hessenberg form have a particular structure that
is not captured by the condition numbers.

The coefficients computed with Algorithms 1 and 2 are denoted by $c_k^{alg1}$
and $c_k^{alg2}$, respectively, while the coefficients computed
by \texttt{poly} are denoted by $c_k^{poly}$. Furthermore,
we distinguish the characteristic polynomials of different matrices
by using $c_k(X)$ for the $k$th coefficient of the characteristic
polynomial of the matrix $X$.

\begin{figure}
\begin{center}
\resizebox{3in}{!}
{\includegraphics{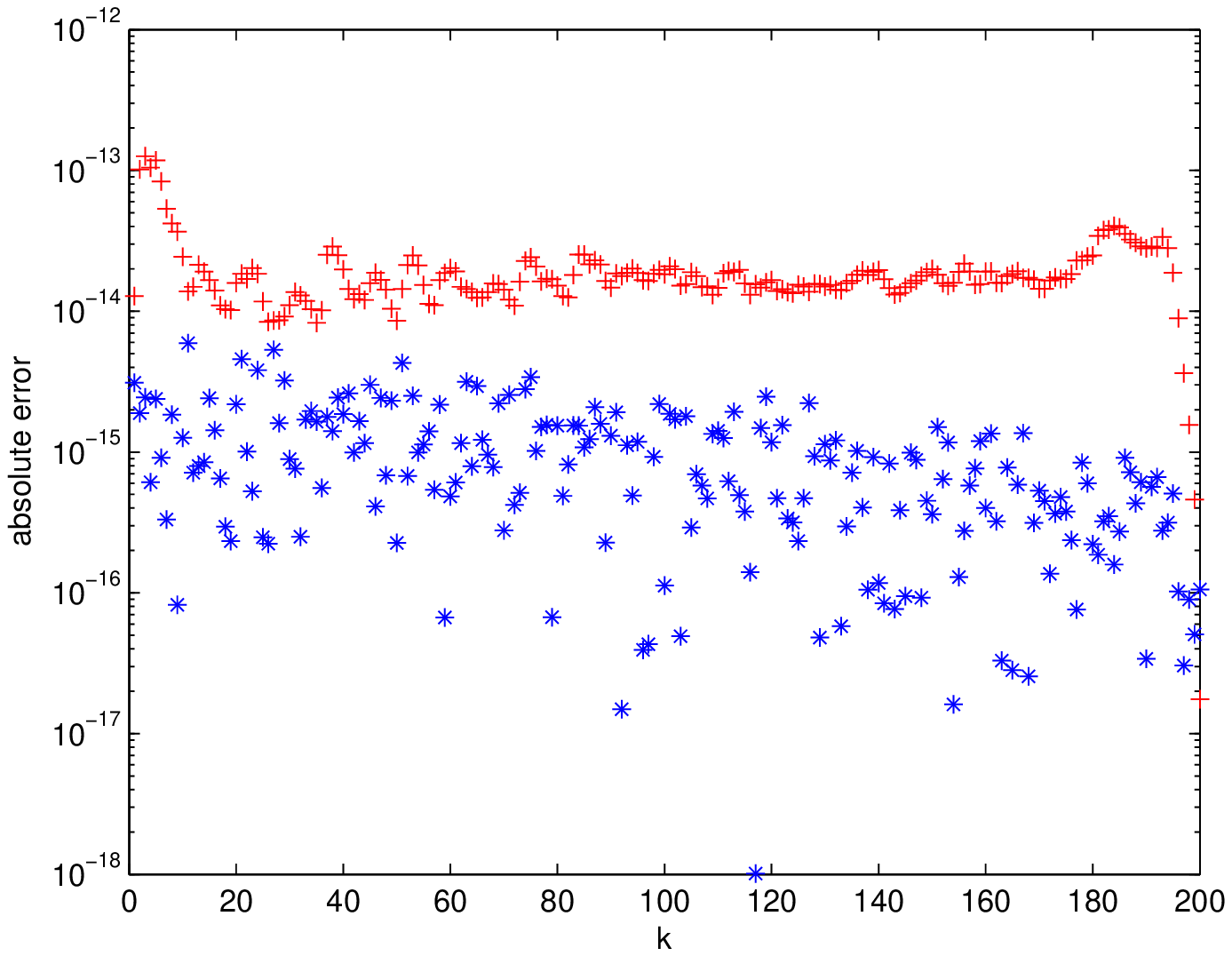}}
\end{center}
\caption{Forsythe Matrix. Lower (blue) curve: 
Absolute errors $|c_k^{alg2}(F)-c_k(F)|$
of the coefficients computed by Algorithm 2. Upper (red) curve:
Running error bounds $\rho_k$ from Corollary \ref{c_runerr}.}
\label{f_forsythe3}
\end{figure}

\begin{figure}
\begin{center}
\resizebox{3in}{!}
{\includegraphics{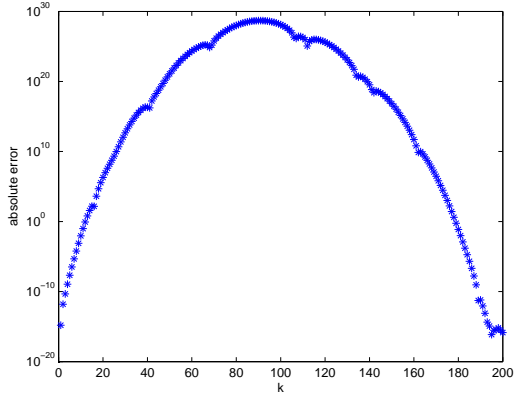}}
\end{center}
\caption{Forsythe Matrix. Coefficients $c_k^{poly}(F)$ computed by
\texttt{poly}. The exact coefficients are $c_k=0$, $1\leq k\leq 199$.}
\label{f_forsythe3poly}
\end{figure}

\subsection{The Forsythe Matrix}
This example illustrates that Algorithm 2 can compute the coefficients
of a highly nonsymmetric matrix more accurately than \texttt{poly}, and that
the running error bounds from Corollary \ref{c_runerr} approximate the 
roundoff error from Algorithm 2 well.

We choose a $n\times n$ Forsythe matrix, which is a 
perturbed Jordan block of the form 
\begin{eqnarray}\label{e_forsythe}
F_2=\begin{pmatrix}
0&1& & & \\
 & \ddots& \ddots & &  \\
 & & 0 & 1& \\
\nu & &  & 0&  \end{pmatrix},
\qquad \mathrm{where} \quad \nu=10^{-10},
\end{eqnarray}
with characteristic polynomial $p(\lambda)=\lambda^n-\nu$.
Then we perform
an orthogonal similarity transformation $F_1=QF_2Q^T$,
where $Q$ is an orthogonal matrix obtained from the QR decomposition
of a random matrix.  The orthogonal similarity transformation 
to upper Hessenberg form $F$ is produced by $F=$ \texttt{hess}$(F_1)$.

We applied Algorithm 2 to a matrix $F$ of order $n=200$.
Figure \ref{f_forsythe3} shows that Algorithm 2 produces absolute errors 
of about $10^{-15}$, and that the running error bounds from Corollary
\ref{c_runerr} approximate the roundoff error from Algorithm 2 well.
In contrast, the absolute errors produced by \texttt{poly} are huge, 
as Figure \ref{f_forsythe3poly} shows.

\begin{figure}
\begin{center}
\resizebox{3in}{!}
{\includegraphics{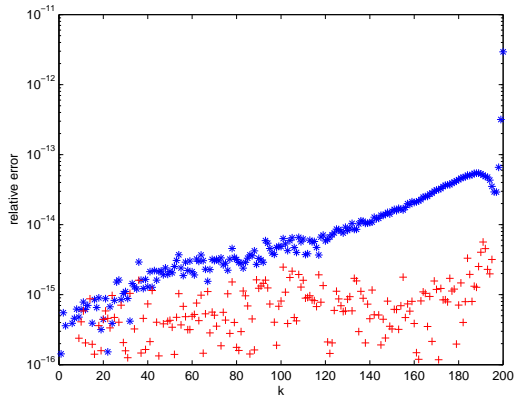}}
\end{center}
\caption{Hansen's Matrix. Upper (blue) curve: 
Relative errors $|c_k^{poly}(H)-c_k(H)|/|c_k(H)|$ of coefficients
computed by \texttt{poly}. Lower (red) curve: Relative errors 
$|c_k^{alg1}(H)-c_k(H)|/|c_k(H)|$ of coefficients computed by Algorithm 1.
}\label{f_hansen3rel}
\end{figure}

\begin{figure}
\begin{center}
\resizebox{3in}{!}
{\includegraphics{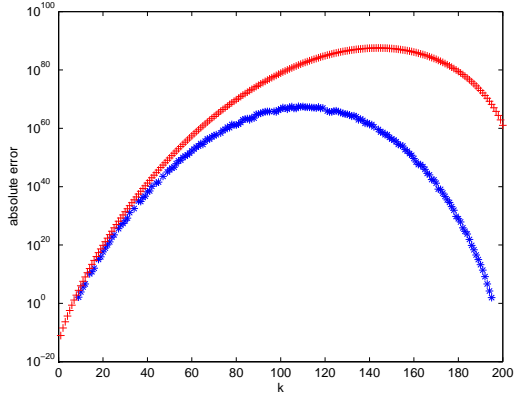}}
\end{center}
\caption{Hansen's Matrix. Lower (blue) curve: Absolute errors 
$|c_k^{alg1}(H)-c_k|$ in the coefficients computed by Algorithm 1.
Upper (red) curve: Running error bounds $\phi_k$ from 
Corollary \ref{c_srunerr}.}\label{f_hansen3abs}
\end{figure}

\subsection{Hansen's Matrix}
This example illustrates that Algorithm 1 can compute the characteristic
polynomial of a symmetric positive definite matrix to machine precision.

Hansen's matrix \cite[p 107]{Han63} is a rank one perturbation of a $n\times n$ 
symmetric tridiagonal Toeplitz matrix,
$$H=\begin{pmatrix}
1&-1& & \\
 -1&2& \ddots & \\
 &\ddots  & \ddots &-1 \\
 & & -1& 2
\end{pmatrix}.$$
Hansen's matrix is positive definite, and 
the coefficients of its characteristic polynomial are
\begin{eqnarray*}
c_{n-k+1}(H)=(-1)^{n-k+1} {n+k-1\choose n-k+1}, \qquad 1\leq k\leq n.
\end{eqnarray*}
Figure \ref{f_hansen3rel} illustrates for $n=200$ that the
Algorithm 1 computes the coefficients to machine precision, and
that later coefficients have higher relative 
accuracy than those computed by \texttt{poly}.
With regard to absolute errors, Figure \ref{f_hansen3abs} indicates that the 
running error bounds $\phi_j$ from Corollary \ref{c_srunerr} reflect the 
trend of the errors, but the
bounds become more and more pessimistic for larger $k$.

\begin{figure}
\begin{center}
\resizebox{3in}{!}
{\includegraphics{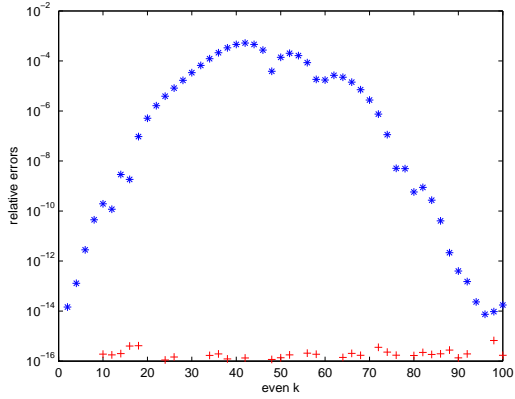}}
\end{center}
\caption{Symmetric Indefinite Tridiagonal Toeplitz Matrix. 
Upper (blue) curve: Relative errors $|c_k^{poly}(T)-c_k(T)|/|c_k(T)|$ 
in the coefficients
computed by \texttt{poly} for even $k$. Lower (red) curve: Relative errors
$|c_k^{alg1}(T)-c_k(T)|/|c_k(T)|$ in the coefficients computed by Algorithm 1
for even $k$.}
\label{f_indef3releven}
\end{figure}

\begin{figure}
\begin{center}
\resizebox{3in}{!}
{\includegraphics{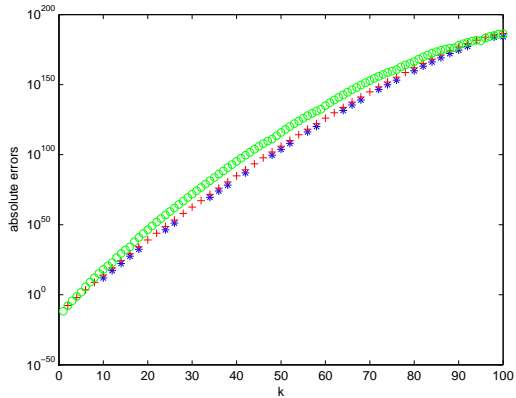}}
\end{center}
\caption{Symmetric Indefinite Tridiagonal Toeplitz Matrix. 
Lower (blue) curve: Absolute errors $|c_k^{alg1}(T)-c_k(T)|$ in the 
coefficients computed by Algorithm 1. Middle (red) curve: Running
error bounds $\phi_j$ from Corollary \ref{c_srunerr}. Upper (green) curve: 
Absolute errors $|c_k^{poly}(T)-c_k(T)|$ in the coefficients 
computed by \texttt{poly}.}
\label{f_indef3abs}
\end{figure}

\subsection{Symmetric Indefinite Toeplitz Matrix}
This example illustrates that Algorithm 1 can compute the characteristic 
polynomial of a symmetric indefinite matrix to high relative
accuracy, and that the running error bounds in Corollary \ref{c_srunerr}
capture the absolute error well.

The matrix is a $n\times n$ symmetric indefinite tridiagonal Toeplitz matrix
$$T=\begin{pmatrix}
0&100& & \\
 100& \ddots & \ddots & &  \\
 & \ddots& 0 & 100\\
 &  & 100& 0 \end{pmatrix},$$
where the coefficients with index are zero, i.e. $c_{2j-1}(T)= 0$ for $j\geq 1$.

For $n=100$ we obtained the exact coefficients $c_k(T)$ with
\texttt{sym2poly(poly(sym(T)))} from MATLAB's symbolic toolbox.
Algorithm 1 computes the coefficients with odd index exactly, i.e.
$c_{2j-1}^{(i)}(T)= 0$ for $j\geq 1$ and $1\leq i\leq n$. In contrast,
as Figure \ref{f_indef3abs} shows, the coefficients
computed by \texttt{poly} can have magnitudes as large $10^{185}$.

Figure \ref{f_indef3releven} illustrates that Algorithm 1
computes the coefficients $c_{2j}(T)$ with even index to machine
precision, while the coefficients computed with \texttt{poly} have relative
errors that are many magnitudes larger.
Figure \ref{f_indef3abs} also shows that the running error bounds 
approximate the true absolute error very well. What is not visible
in Figure \ref{f_indef3abs}, but what one can show from Theorems
\ref{error_1} and \ref{error_2} is that $\phi_{2j-1}=0$. Hence the
running error bounds recognize that $c_{2j-1}$ are computed exactly.

\begin{figure}
\begin{center}
\resizebox{3in}{!}
{\includegraphics{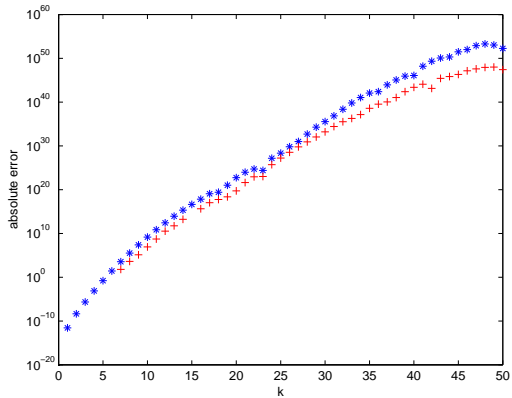}}
\end{center}
\caption{Frank Matrix. Upper (blue) curve: Absolute errors 
$|c_k^{poly}(U)-c_k(U)|$ in the coefficients computed by \texttt{poly}.
Lower (red)  curve: Absolute errors $|c_k^{alg2}(U)-c_k(U)|$ in the 
coefficients computed by Algorithm 2.}
\label{f_frank3abs}
\end{figure}

\begin{figure}
\begin{center}
\resizebox{3in}{!}
{\includegraphics{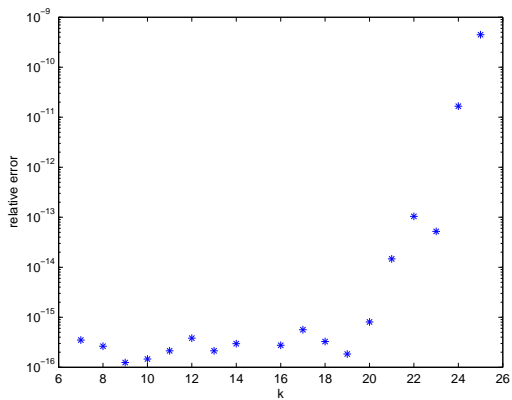}}
\end{center}
\caption{Frank Matrix. Relative errors 
$|c_k^{alg2}(U)-c_k(U)|/|c_k(U)|$ in the first 25
coefficients computed by Algorithm 2.}
\label{f_frank3rel}
\end{figure}

\subsection{Frank Matrix}
This example shows that Algorithm 2 is at least as accurate, if not more
accurate than \texttt{poly} for matrices with ill conditioned polynomial
coefficients.

The Frank matrix $U$ is an upper Hessenberg matrix with 
determinant 1 from MATLAB's \texttt{gallery} command of test matrices.
The coefficients of the characteristic polynomial 
appear in pairs, in the sense that $c_k(U)=c_{n-k}(U)$.
For a Frank matrix of order $n=50$, we used MATLAB's toolbox to
determine the exact coefficients $c_k(U)$ with the command
\texttt{sym2poly(poly(sym(U)))}. Figure \ref{f_frank3abs} illustrates
that Algorithm 2 computes the coefficients at least as accurately as
\texttt{poly}. In fact, as seen in Figure \ref{f_frank3rel}, Algorithm 2
computes the first 20 coefficients to high relative accuracy.

\begin{figure}
\begin{center}
\resizebox{3in}{!}
{\includegraphics{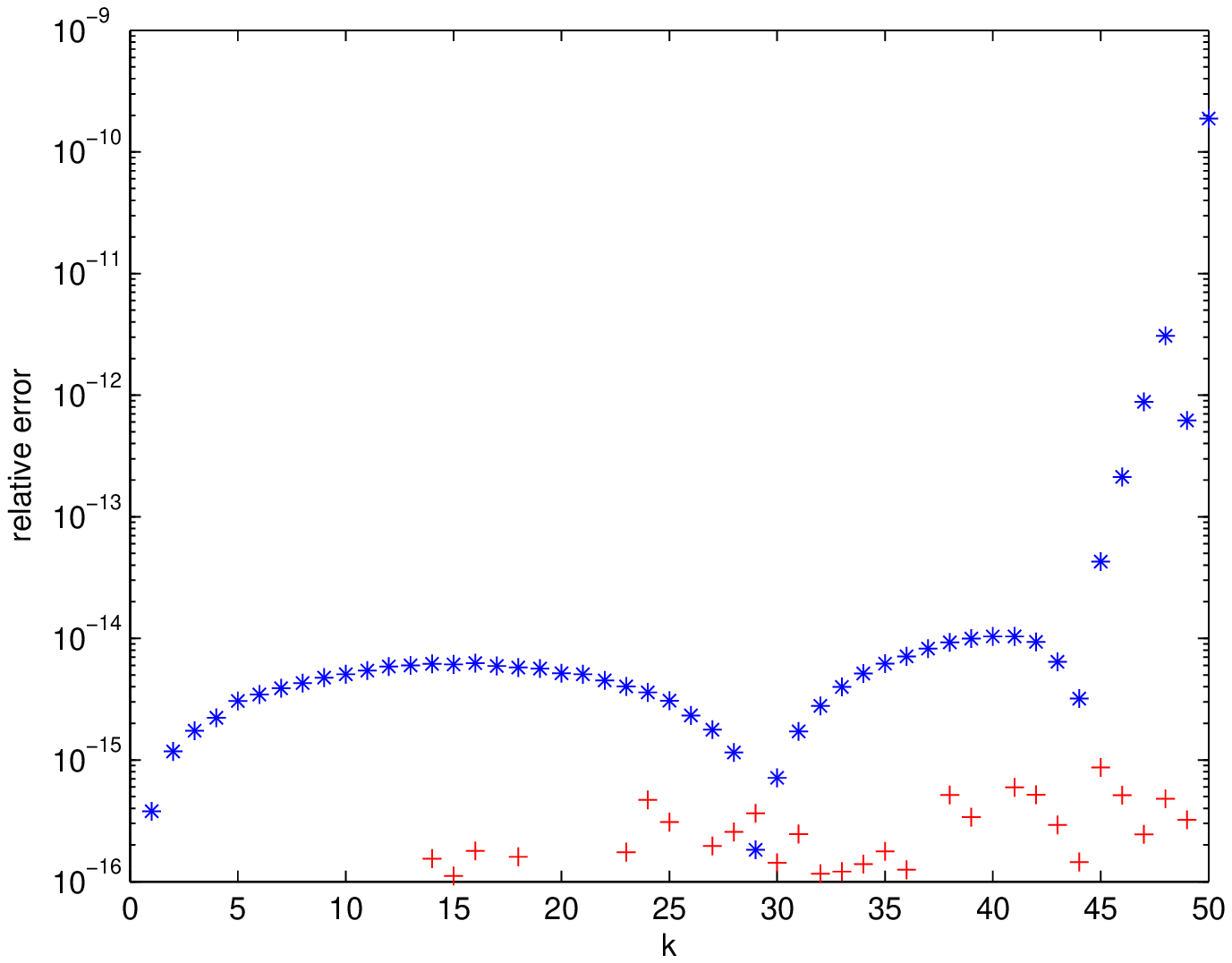}}
\end{center}
\caption{Transposed Chow Matrix. Upper (blue) curve: Relative errors 
$|c_k^{poly}(C^T)-c_k(C^T)|/|c_k(C^T)|$ in the coefficients computed 
by \texttt{poly}.
Lower (red) curve: Relative errors $|c_k^{alg2}(C^T)-c_k(C^T)|/|c_k(C^T)|$ 
in the coefficients computed by Algorithm 2.}
\label{f_chow3}
\end{figure}

\begin{figure}
\begin{center}
\resizebox{3in}{!}
{\includegraphics{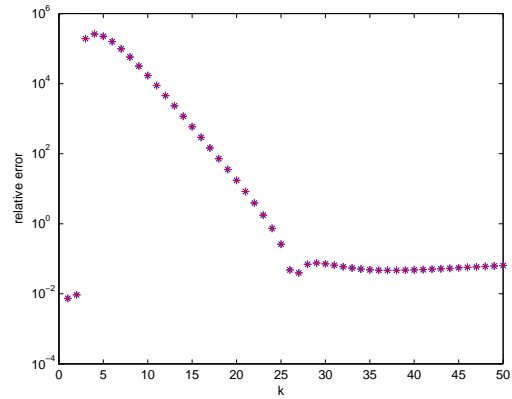}}
\end{center}
\caption{Chow Matrix. Blue curve: Relative errors 
$|c_k^{poly}(C)-c_k(C)|/|c_k(C)|$ 
in the coefficients computed by \texttt{poly}.
Red curve: Relative errors $|c_k^{alg2}(C)-c_k(C)|/|c_k(C)|$
in the coefficients computed by Algorithm 2. The two curves are 
virtually indistinguishable.}
\label{f_chowt3}
\end{figure}

\subsection{Chow Matrix}
This example illustrates that the errors in the reduction to Hessenberg
form can be amplified substantially when the coefficients of the
characteristic polynomial are illconditioned.

The Chow matrix is a matrix that is Toeplitz as well
as lower Hessenberg from MATLAB's \texttt{gallery} command of test matrices.
Our version of the transposed Chow matrix is an upper Hessenberg matrix
with powers of 2 in the leading row and trailing column,
$$C^T=\begin{pmatrix}3 &     4& \ldots      &2^n\\
                   1& \ddots& \ddots& \vdots\\
                    & \ddots& 3     & 4\\
                    &       & 1     &3\end{pmatrix}.$$
As before, we computed the exact coefficients with MATLAB's symbolic 
toolbox. Figure \ref{f_chow3} illustrates that Algorithm 2 computes all 
coefficients $c_k(C^T)$ to high relative accuracy for $n=50$.
In contrast,  the relative accuracy
of the coefficients computed by \texttt{poly} deteriorates markedly as $k$
becomes larger.

However, if we compute instead the characteristic polynomial of $C$, then
a preliminary reduction to upper Hessenberg form is necessary. 
Figure \ref{f_chowt3} illustrates that the computed coefficients have hardly any
relative accuracy to speak of, and only the trailing coefficients
have about 1 significant digit. The loss of accuracy occurs
because the errors in the reduction to Hessenberg form are amplified by the 
condition numbers of the coefficients, as the absolute
bound in Theorem \ref{t_nonsymerr} suggests. Unfortunately, 
in this case, the condition numbers in Theorem \ref{t_nonsymerr} 
are too pessimistic to predict
the absolute error of Algorithm 2.

\subsection*{Acknowledgements} We thank Dean Lee for helpful discussions.

\end{document}